\documentclass[11pt]{amsart}
\usepackage[utf8]{inputenc}
\usepackage{enumerate,enumitem,amsmath,amssymb,mathrsfs,stmaryrd}
\usepackage{latexsym,epsf,graphicx,comment,appendix}
\usepackage{MnSymbol}

\usepackage{tikz}
\usetikzlibrary{arrows.meta}
\usetikzlibrary{automata}
\usetikzlibrary{chains}

\usepackage[english]{babel}
\usepackage{color}
\usepackage{times}
\usepackage[T1]{fontenc}

\usepackage{amsmath}
\usepackage{dsfont}
\usepackage{amsfonts}

\newfont{\bb}{msbm10 at 12pt}
\newfont{\tbb}{msbm10 at 8pt}

\def\n{\hbox{\bb N}}

\baselineskip 
\usepackage {amsmath}
\usepackage {amsthm}
\usepackage{times}
\usepackage{amscd}
\usepackage{epsf}

\numberwithin{equation} {subsection}

\begin{document}
\mbox{}\vspace{0.2cm}\mbox{}

\providecommand{\keywords}[1]
{
  \small	
  \textbf{\textit{Keywords---}} #1
}

\

\theoremstyle{plain}\newtheorem{lemma}{Lemma}[subsection]
\theoremstyle{plain}\newtheorem{proposition}{Proposition}[subsection]
\theoremstyle{plain}\newtheorem{theorem}{Theorem}[subsection]

\theoremstyle{plain}\newtheorem{theorem-intr}{Theorem}[]
\theoremstyle{plain}\newtheorem*{theorem*}{Theorem}
\theoremstyle{plain}\newtheorem*{main theorem}{Main Theorem} 
\theoremstyle{plain}\newtheorem*{lemma*}{Lemma}
\theoremstyle{plain}\newtheorem*{claim}{Claim}

\theoremstyle{plain}\newtheorem{example}{Example}[subsection]
\theoremstyle{plain}\newtheorem{remark}{Remark}[subsection]
\theoremstyle{plain}\newtheorem{corollary}{Corollary}[subsection]
\theoremstyle{plain}\newtheorem{corollary-intr}{Corollary}[]
\theoremstyle{plain}\newtheorem*{corollary-A}{Corollary}
\theoremstyle{plain}\newtheorem{definition}{Definition}[subsection]
\theoremstyle{plain}\newtheorem{acknowledge}{Acknowledgment}
\theoremstyle{plain}\newtheorem{conjecture}{Conjecture}

\begin{center}
\rule{15cm}{1.5pt} \vspace{.4cm}

{\bf\Large On the asymptotic behavior of $p$-superharmonic functions at singularities} 
\vskip .3cm

Huajie Liu, Shiguang Ma, Jie Qing, and Shuhui Zhong

\vspace{0.3cm} 
\rule{15cm}{1.5pt}
\end{center}


\title{}

\begin{abstract} In this paper we develop the $p$-thinness and the $p$-fine topology for the asymptotic behavior of $p$-superharmonic functions at singular points.
We consider these as extensions of earlier works on superharmonic functions in dimension 2 in \cite{AH73},  on the Riesz
and Log potentials in higher dimensions in \cite[Section 2.5 and 2.6]{Mi96}, and on $p$-harmonic functions in \cite{KV}. 
It is remarkable that, contrary to the above cases, 
the $p$-thinness for the singular behavior (cf. Definition \ref{Def:quasi-p-thin}) differs from the $p$-thinness for continuity by the Wiener criterion for 
$p$-superharmonic functions. As applications of asymptotic estimates of $p$-superharmonic functions, we also obtain asymptotic estimates of solutions 
to a class of fully nonlinear elliptic equations, which go beyond \cite[Theorem 3.6]{Lab02}. This paper grows out of our recent papers on the potential theory 
in conformal geometry \cite{MQ21, MQ021, MQ22, MQ23}.
\end{abstract}



\maketitle

\section{Introduction}\label{Sec:intr}

In this paper we are concerned with the asymptotic behavior of $p$-superharmonic functions at singular points. 
We develop the $p$-thinness (cf. Definition \ref{Def:quasi-p-thin}) 
and the induced  $p$-fine topology for the asymptotic behavior of $p$-superharmonic functions at 
singular points. The notion of the fine topology was first introduced to potential theory by Henri Cartan in 1940 as the weakest topology on Euclidean space
$\mathbb{R}^n$ making all superharmonic functions continuous. Such idea also works for identifying the topology to get asymptotics at singularities seen 
in \cite{AH73} and \cite[Section 2.5 and 2.6]{Mi96}, which has geometric applications in \cite{MQ21, MQ021,MQ22, MQ23} following \cite{Hu57, AH73}. 
It is remarkable that, contrary to the single notion of thinness and the fine topology for 
the Riesz potentials and Log potentials at regular and singular points, the $p$-thinness for 
the singular behavior is different from the $p$-thinness for continuity by the Wiener criterion on $p$-superharmonic functions. 
The $p$-thinness for continuity by the Wiener criterion \eqref{Equ:p-thin-wiener}
has been confirmed to induce the fine topology for continuity of $p$-superharmonic functions for $p\in (1, n]$ in the seminal paper \cite{KM94} (see also 
\cite{HeKi, KM92, HKM}). We are able to confirm the $p$-thinness for the singular behavior by Definition \ref{Def:quasi-p-thin}
indeed induces the fine topology for taking limits to get the asymptotic of $p$-superharmonic functions at singular points for $p\in (1, n]$.


\subsection{On the Wolff potentials and the $p$-fine topology}

In nonlinear potential theory, one important tool for the study of $p$-Laplace equations is the Wolff potential that in some way substitutes the Riesz potential
and the Log potential in linear potential theory
\begin{equation*}
W^{\mu}_{1,p}(x,\, r)=\int_0^r\left( \frac{\mu(B(x,t))}{t^{n-p}}\right)^{\frac 1 {p-1}}\frac {dt}{t}
\end{equation*}
for a nonnegative Radon measure $\mu$ and $p\in (1, n]$.  We first focus on the asymptotic behavior of the Wolff potentials at singular points. 
After the introduction of 
the $p$-thinness for the singular behavior in Definition \ref{Def:quasi-p-thin}, we establish 

\begin{theorem-intr}\label{Thm:asymptotic-wolff-intr}
Let $\mu$ be a nonnegative finite Radon measure in $\Omega$ and $B(x_0, 3r_0)\subset \Omega$. 
Then there is a subset $E$ that is $p$-thin for the singular behavior at $x_0$ such that
\begin{equation*}
\lim_{x\to x_0 \text{ and } x\notin E } |x-x_0|^\frac {n-p}{p-1} W_{1,p}^{\mu}(x, \, r_0) = \frac{p-1}{n-p} \mu(\{x_0\})^\frac 1{p-1}
\end{equation*}
for $p\in (1, n)$. Similarly, there is a subset $E$ that is $n$-thin for the singular behavior at $x_0$ such that
\begin{equation*}
\lim_{x\to x_0 \text{ and } x\notin E} \frac{W_{1,n}^{\mu}(x, \, r_0)}{\log\frac 1{|x-x_0|}} = \mu(\{x_0\})^\frac 1{n-1}.
\end{equation*}
\end{theorem-intr}

To confirm the $p$-thinness for the singular behavior in Definition \ref{Def:quasi-p-thin} indeed induces the $p$-fine topology for the singular behavior we
also show

\begin{theorem-intr}\label{Thm:fine-topology-intr}
Suppose that $E$ is a subset that is $p$-thin for the singular behavior at the origin
according to Definition \ref{Def:quasi-p-thin} for $p\in (1, n]$. And suppose that the origin is in $\bar E\setminus E$. Then, when $p\in (1, n)$, 
there is a Radon measure $\mu$ in a neighborhood of the origin such that, for some fixed $r_0>0$, 
$$
\lim_{x\to 0 \text{ and } x\in E} |x|^\frac {n-p}{p-1}W^\mu_{1, p}(x, r_0) = \infty.
$$
Similarly, when $p=n$, there is a Radon measure $\mu$ in a neighborhood of the origin such that, for some fixed $r_0>0$, 
$$
\lim_{x\to 0 \text{ and } x\in E} \frac {W^\mu_{1, n}(x, r_0)}{\log\frac 1{|x|}} = \infty.
$$
\end{theorem-intr}

Combining Theorem \ref{Thm:asymptotic-wolff-intr} and \ref{Thm:fine-topology-intr} we conclude

\begin{corollary-intr}\label{Cor:wolff-p-thin-intr} A subset $E\subset \mathbb{R}^n$ is 
$p$-thin for the singular behavior at the origin if and only if there is a Radon measure $\mu$ such that
$$
\lim_{x\to 0 \text{ and } x\in E} \frac {W^\mu_{1, p} (x, r_0)}{|x|^{-\frac {n-p}{p-1}}} > \liminf_{x\to 0}\frac {W^\mu_{1, p} (x, r_0)}{|x|^{-\frac {n-p}{p-1}}}
$$
for $p\in (1, n)$. Similarly, a subset $E\subset \mathbb{R}^n$ is 
$n$-thin for the singular behavior at the origin if and only if there is a Radon measure $\mu$ such that
$$
\lim_{x\to 0 \text{ and } x\in E} \frac {W^\mu_{1, p} (x, r_0)}{\log \frac 1{|x|}} > \liminf_{x\to 0}\frac {W^\mu_{1, p} (x, r_0)}{\log \frac 1{|x|}}.
$$
\end{corollary-intr}

\subsection{On $p$-superharmonic functions and applications}
As a straightforward consequence of Theorem \ref{Thm:fine-topology-intr} together with the existence Lemma \ref{Lem:existence-p-sh} 
(\cite[Theorem 2.4]{KM92}) and the fundamental estimate Theorem \ref{Equ: main-use-wolff} (cf. \cite[Theorem 1.6]{KM94}), 
we first have

\begin{theorem-intr}\label{Thm:infty-p-intr}
Suppose that $E$ is a subset that is $p$-thin for the singular behavior at the origin
according to Definition \ref{Def:quasi-p-thin} for $p\in (1, n]$. And suppose that the origin is in $\bar E\setminus E$. Then, when $p\in (1, n)$, 
there is a $p$-superharmonic function $u$ in a neighborhood of the origin such that
$$
\lim_{x\to 0 \text{ and } x\in E} |x|^\frac {n-p}{p-1} u(x)  = \infty.
$$
Similarly, when $p=n$, there is a $n$-superharmonic function $u$ in a neighborhood of the origin such that
$$
\lim_{x\to 0 \text{ and } x\in E} \frac {u(x)}{\log\frac 1{|x|}} = \infty.
$$
\end{theorem-intr}

Theorem \ref{Thm:infty-p-intr} suggests that one can only derive the asymptotic behavior of $p$-superharmonic functions at singularities 
under the $p$-fine topology in general. Indeed we are able to complete the asymptotic estimates for $p$-superharmonic functions at singularities 
based on the rescaling approach in \cite{KV, Vn17} (see also \cite{MQ21, MQ021}) and the asymptotic estimates on the Wolff potentials at singularities.  
Our work here is to realize the uniqueness of all possible sequential limits and then use the precise asymptotic behavior of the 
Wolff potential at singularities with respect to the $p$-fine topology in Theorem \ref{Thm:asymptotic-wolff-intr} to derive 
the asymptotic behavior of $p$-superharmonic functions at singularities with respect to the $p$-fine topology. 
 
\begin{theorem-intr}\label{Thm:main-asymptotic-intr}
Suppose that $u$ is a nonnagetive $p$-superharmonic function in $\Omega\subset\mathbb{R}^n$ satisfying
$$
-\Delta_p u = \mu \text{ in $\Omega$}
$$
for a nonnegative finite Radon measure $\mu$ on $\Omega$ and $p\in (1, n]$. Then, for $x_0\in\Omega$, there is a subset $E$ that is $p$-thin for singular
behavior at $x_0$ such that
\begin{equation}\label{Equ:p-super-asymptotic}
\lim_{x\to x_0\text{ and } x\notin E}  \frac{u(x)}{G_p(x, x_0)} = m = \left\{\aligned \frac{p-1}{n-p} (\frac {\mu(\{0\})}{|\mathbb{S}^{n-1}|})^\frac 1{p-1}
& \text{ when $p\in (1, n)$}\\ (\frac {\mu(\{0\})}{|\mathbb{S}^{n-1}|})^\frac 1{n-1} & \text{ when $p=n$}\endaligned\right.,
\end{equation} 
where
\begin{equation}\label{Equ:G-p}
G_p(x, x_0) =\left\{\aligned |x-x_0|^{-\frac {n-p}{p-1}} & \text{ when } p\in (1, n)\\
- \log |x-x_0| & \text{ when } p=n\endaligned\right..
\end{equation}
Moreover $u(x) \geq mG_p(x, x_0) - c_0$ for some $c_0$ and all $x$ in a neighborhood of $x_0$.
\end{theorem-intr}

Consequently, parallel to Corollary \ref{Cor:wolff-p-thin-intr}, we conclude,

\begin{corollary-intr} A subset $E\subset \mathbb{R}^n$ is 
$p$-thin for the singular behavior at the origin if and only if there is a $p$-superharmonic function $u(x)$ such that
$$
\lim_{x\to 0 \text{ and } x\in E} \frac {u(x)}{|x|^{-\frac {n-p}{p-1}}} > \liminf_{x\to 0}\frac {u(x)}{|x|^{-\frac {n-p}{p-1}}}.
$$
for $p\in (1, n)$. Similarly, a subset $E\subset \mathbb{R}^n$ is 
$n$-thin for the singular behavior at the origin if and only if there is a $n$-superharmonic function $u(x)$ such that
$$
\lim_{x\to 0 \text{ and } x\in E} \frac {u(x)}{\log \frac 1{|x|}} > \liminf_{x\to 0}\frac {u (x)}{\log \frac 1{|x|}}.
$$
\end{corollary-intr}

As applications, we derive the asymptotics of solutions to a class of fully nonlinear elliptic equations (cf. Corollary 
\ref{Cor:app-fully-nonlinear}) and extend \cite[Theorem 3.6]{Lab02} (please see Corollary \ref{Cor:sigma-k-asymptotic}).

\begin{theorem-intr}\label{Thm:sigma-k-intr}
Suppose that $u$ is nonnegative and that $u\in C^2 (\Omega\setminus S)$ for a
compact subset $S$ inside a bounded domain $\Omega$ in $\mathbb{R}^n$. And suppose 
$\lim_{x\to S}u(x) = +\infty.$
Assume $-\lambda (D^2 u(x)) \in \Gamma^k$ for $ 1\leq k \leq \frac n2$, where $\Gamma^k=\{\lambda\in\mathbb{R}^n: \sigma_1(\lambda)\geq 0,
\cdots, \sigma_k(\lambda)\geq 0\}$. Then $S$ is of Hausdorff dimension not greater than $n-p_\Gamma$ and, 
for $x_0\in S$, there are a subset $E$ that is $p_{\Gamma^k}$-thin 
for the singular behavior at $x_0$ and a nonnegative number $m$ such that
$$
\lim_{x\to x_0 \text{ and } x\notin E}  \frac {u(x)}{\mathcal{G}^k(x, x_0)} = m,
$$
where
$$
p_{\Gamma^k}  = \frac {n(k-1)}{n-k} + 2 \in [2, n]
$$
for $1\leq k \leq \frac n2$ and
$$
\mathcal{G}^k(x, x_0) = \left\{\aligned |x-x_0|^{2 - \frac nk} & \text{ when } 1\leq k < \frac n2\\
- \log |x-x_0| & \text{ when } k = \frac n2\endaligned\right..
$$
Moreover $u(x) \geq m \mathcal{G}^k(x, x_0) - c_0$ in some neighborhood of $x_0$. 
\end{theorem-intr}

 
 \subsection{Organization} The rest of the paper is organized as follows: In Section \ref{Sec:asymptotic-wolff}, we recall definitions and basic facts in nonlinear
 potential theory. Moreover, we prove Theorem \ref{Thm:asymptotic-wolff-intr} and Theorem \ref{Thm:fine-topology-intr} on the asymptotic behavior of the
 Wolff potentials at singularities under the $p$-fine topology. In Section \ref{Sec:asymptotics-p-superharmonic}, we first prove 
 Theorem \ref{Thm:main-asymptotic-intr} and then prove Theorem \ref{Thm:sigma-k-intr} as applications of Theorem \ref{Thm:main-asymptotic-intr} 
 following Lemma \ref{Lem:calculate-p-index} and the calculation of $p_{\Gamma^k}$.


\section{On nonlinear potential theory and $p$-Laplace equations}\label{Sec:asymptotic-wolff}

In this section, we start with definitions and collect some relevant facts in the potential theory for the study of $p$-superharmonic functions. 
We pay particular attention to the distinct definitions of $p$-thinness for continuity and the singular behavior (cf. Definition \ref{Def:p-thin-wiener} 
and Definition \ref{Def:quasi-p-thin}). 
We then prove the analogue of \cite[Theorem 5.3 and Theorem 6.3 in Chapter 2]{Mi96} on the asymptotic behavior of the Wolff 
potentials at singularities (cf. Theorem \ref{Thm:wolff potential upper bdd-singualarity}).  We are also able to confirm that the notion 
of $p$-thinness for the singular behavior in Definition \ref{Def:quasi-p-thin} 
indeed induces the $p$-fine topology by establishing the analogue of \cite[Theorem 5.4 in Section 2.5]{Mi96} (cf. Theorem \ref{Thm:converse-p-thin}). 
Our main references for background in nonlinear potential theory and $p$-Laplace equations are \cite{HeKi, KM92, HKM, KM94, Mi96, AH96, Lind06, PV08, Vn17}.

\subsection{Definitions and basics in nonlinear potential theory}\label{Subsec:intr-p-wolff}

First we want to give definitions of $p$-harmonic and  $p$-superharmonic functions. In this subsection we always assume $1<p\leq n$ unless specified otherwise.
It is good to recall what is the $p$-Laplace operator
\begin{equation*}
\Delta_p u = \text{div}(|\nabla u|^{p-2}\nabla u).
\end{equation*}

\begin{definition}\label{Def:p-harmonic} (\cite[Definition 2.5]{Lind06})
We say that $u\in W^{1,p}_{\rm{loc}}(\Omega)$ is a weak solution of the $p$-harmonic equation in $\Omega$, if 
\begin{equation*}
\int \langle |\nabla u|^{p-2}\nabla u,\nabla \eta \rangle dx=0
\end{equation*}
for each $\eta\in C_0 ^{\infty}(\Omega)$. If, in addition, $u$ is continuous, we then say $u$ is a $p$-harmonic function.
\end{definition}

\begin{definition}\label{Def:p-superhar} (\cite[Definiton 5.1]{Lind06})
A function $v:\Omega\to (-\infty,\infty]$ is called $p$-superharmonic in $\Omega$, if 
\begin{itemize}
\item $v$ is lower semi-continuous in $\Omega$;
\item $v\not\equiv\infty$ in $\Omega$;
\item for each domain $D\subset\subset\Omega$ the comparison principle holds, that is, 
if $h\in C(\bar{D})$ is $p$-harmonic in $D$ and $h|_{\partial D}\le v|_{\partial D}$, then $h\le v$ in D. 
\end{itemize}
\end{definition}

As stated in \cite[Theorem 2.1]{KM92}, for any $p$-superharmonic function $u$ in $\Omega$, there is a nonnegative Radon 
measure $\mu$ in $\Omega$ such that
\begin{equation}\label{Equ:p-laplace-measure}
-\Delta_p u=\mu.
\end{equation}
We want to mention the following local summability for $p$-superharmonic functions 
that will be used in the proof of Lemma \ref{Lem:point-charge} in Section \ref{Subsec:arsove-huber}.

\begin{lemma}\label{Lem:l^p-estimate} (\cite[Theorem 5.11 and 5.15]{Lind06}) Let $p\in (1, n]$ and $u$ is a nonnegative 
$p$-superharmonic function satisfying \eqref{Equ:p-laplace-measure} for a finite nonnegative Radon measure $\mu$ on $\Omega$. 
And let $B(2R)\subset \Omega$ be a ball. Then
\begin{equation}\label{Equ:u-l^q}
(\int_{B(R)} u^s dx)^\frac 1s \leq C (n, p, s, R) {\rm ess}\inf_{B(R)} u
\end{equation}
for all $s\in (0, \frac{n(p-1)}{n-p})$ and 
\begin{equation}\label{Equ:gradient-l^q}
(\int_{B(R)} |\nabla u|^q dx)^\frac 1q \leq C (n, p, q, R) \text{ess}\inf_{B(R)} u
\end{equation}
for any $q\in (0, \frac {n(p-1)}{n-1})$.
\end{lemma}

For more about $p$-superharmonic functions and nonlinear potential theory, we refer to  \cite{HeKi, KM92, HKM, KM94, PV08, AH96, Lind06, Vn17} and
references therein. One important tool is the Wolff potentials which in some way substitute the Riesz potentials in linear potential theory

\begin{equation}\label{Equ:wolff-potential}
W^{\mu}_{1,p}(x,\, r)=\int_0^r \left( \frac{\mu(B(x,t))}{t^{n-p}}\right)^{\frac 1 {p-1}}\frac {dt}{t}
\end{equation}
for any nonnegative Radon measure $\mu$ and $p\in (1, n]$. 
\\

We follow the variational approach to introduce the $p$-capacity in \cite[Chapters 2 and 4]{HKM} and \cite[Section 3.1]{KM94}.

\begin{definition}\label{Def:p-capacity} (\cite[Section 3.1]{KM94}) 
For a compact subset $K$ of a domain $\Omega$ in $\mathbb{R}^n$, we define
\begin{equation}\label{Equ:p-capacity-K}
cap_p(K,\Omega)=\inf \{\int_{\Omega}|\nabla u|^p dx: \text{ $u\in C_0^{\infty}(\Omega)$ and $u\ge 1$ on $K$}\}.
\end{equation}
Then $p$-capacity for arbitrary subset $E$ of $\Omega$ is 
\begin{equation}\label{Equ:p-capacity}
cap_p(E,\Omega)=\inf_{ E \subset G \text{ \& }
G  \subset\Omega \, \text{open}} \ \ \sup_{K\subset G \, \text{compact}} cap_p(K,\Omega).
\end{equation}
\end{definition}

We would like to mention some basic facts about the capacity $cap_p(E, \Omega)$ which are important for the arguments in this paper. Readers
are referred to \cite{HKM, KM94, AH96}, for example.

\begin{lemma}\label{Lem:subadditivity-Lip}\quad

\begin{enumerate}
\item Countable sub-additivity: 
$$
cap_p(\bigcup_k E_k, \Omega) \leq \sum_k cap_p(E_k, \Omega)
$$
for a countable collection of subsets $\{E_k\}$ of $\Omega$.
\item Scaling property:
$$
cap_p(E_\lambda, \Omega_\lambda) = \lambda^{n-p} cap_p(E, \Omega)
$$
where $A_\lambda = \{\lambda x\in \mathbb{R}^n: x\in A\}$ is the scaled set of $A \subset\mathbb{R}^n$.
\item Lipschitz property: Let 
$$
\Phi:\mathbb{R}^n \to \mathbb{R}^n
$$
be a Lipschitz map, that is, there is a constant $C_0>0$ such that 
$$
|\Phi(x)-\Phi(y)| \leq C_0|x-y| \text{ for all $x, y\in \mathbb{R}^n$}.
$$
Then, there is a constant $C>0$ such that
$$
cap_p (\Phi(E), \mathbb{R}^n) \leq C cap_p(E, \mathbb{R}^n).
$$
\item Positivity:
$$
cap_p(\partial B(0, r), B(0, 2r)) = c(n, p) r^{n-p} >0.
$$
\end{enumerate}
\end{lemma} 

\begin{proof} The proof and references are here: (1) is from \cite[Theorem 2.2]{HKM}; (2) is easily verified; 
(3) is easily been verified based on Definition \ref{Def:p-capacity} 
(see also \cite[Theorem 5.2.1]{AH96}); (4) has been explicitly computed, for instance, \cite[Eample 2.12]{HKM}.
\end{proof}

The capacity in Definition \ref{Def:p-capacity} (cf. \cite{HKM, KM94}) turns out to be equivalent 
to the dual capacity according to \cite[Theorem 3.5]{KM94}, at least for capacitable condensers.
More importantly, we have the following when dealing with capacitary functions and distributions.

\begin{lemma}\label{Lem:dual-capacity} (\cite[Corollary 3.8]{KM94}) Suppose that $\Omega$ is a bounded domain and $E\subset\subset\Omega$.
Let $\mu$ be the capacitary distribution of $E$ in $\Omega$. Then
$$
\mu(\Omega) = cap_p(E, \Omega).
$$
\end{lemma}
   
To construct the capacitary function and distribution of a condenser $(E, \Omega)$, one uses the theory of balayage. For us, the capacitary functions and distributions 
are important for understanding the fine topology for the asymptotic behavior of the Wolff potentials and $p$-superharmonic functions at singularities. We follow
\cite[Chapter 8]{HKM} in the rest of this subsection. First we recall
 
 \begin{definition}\label{Def:reduite} (\cite[Chapter 8]{HKM}) Let $\psi: \Omega\to (-\infty, +\infty]$ be locally bounded from below and 
 $$
 \mathcal{M}^\psi (\Omega) = \{u: u\geq \psi \text{ in $\Omega$, where $u$ is $p$-superharmonic function in $\Omega$}\}.
 $$
 Then the function
 $$
 R^\psi = R^\psi(\Omega) = \inf \mathcal{M}^\psi(\Omega)
 $$
 is called the r\'{e}duite and its lower semicontinuous regularization
 $ \hat R^\psi =\hat R^\psi(\Omega)$ is called the balayage of $\psi$. Here $\hat R^\psi \equiv \infty$ if $\mathcal{M}^\psi=\emptyset$.  
 For a subset $E\subset \Omega$ and  $\psi$ being the characteristic function of $E$ in $\Omega$, we call
 $\hat R_E = \hat R^\psi$ the $p$-potential of $E$ in $\Omega$.
 \end{definition}
 
 Then, as observed in \cite[Chapter 8]{HKM},
 
 \begin{theorem}\label{Thm:balayage-existence} (\cite[Theorem 8.1]{HKM}) The balayage $\hat R^\psi$ is $p$-superharmonic in $\Omega$.
 \end{theorem}
 
 And more importantly, 
 
 \begin{theorem}\label{Thm:capacitary-function} (\cite[Theorem 9.35]{HKM}) Let $E$ be a subset of a bounded domain $\Omega$ with finite capacity 
 $cap_p(E, \Omega)$. And let $\hat R_E$ be the $p$-potential of $E$ in $\Omega$. Then
 \begin{equation}\label{Equ:capacitary-function}
 cap_p(E, \Omega) = \int_\Omega |\nabla \hat R_E|^p dx
 \end{equation}
 where $\hat R_E\in H^{1, p}_0(\Omega)$. Therefore $\hat R_E$ is the capacitary function and 
 $-\Delta_p \hat R_E$ is the capacitary distribution of the subset $E$ in $\Omega$.
 \end{theorem}
 

\subsection{$p$-superharmonic functions and the Wolff potentials}
The fundamental estimates in terms of the Wolff potential in the study of $p$-superharmonic
functions is the following:

\begin{theorem}\label{Thm:main-use-wolff} (\cite[Theorem 1.6]{KM94}) 
Suppose that $u$  is a nonnegative $p$-superharmonic function satisfying \eqref{Equ:p-laplace-measure} for a nonnegative 
finite Radon measure $\mu$ in $B(x, 3r)$. Then
\begin{equation}\label{Equ: main-use-wolff}
c_1 W^{\mu}_{1,p}(x, r)\le u(x)\le c_2(\inf_{B(x, r)}u+W^{\mu}_{1,p}(x, 2r))
\end{equation}
for some constants $c_1(n,p)$ and $c_2(n,p)$ for $p\in (1, n]$.
\end{theorem}

It is often useful to have the following existence result in order to fully use the above estimates (cf. \cite[Theorem 1.6]{KM94}).

\begin{lemma}\label{Lem:existence-p-sh} (\cite[Theorem 2.4]{KM92})
Suppose $\Omega\subset\mathbb{R}^n$ is a bounded domain and $\mu$ is a nonnegative finite Radon measure. 
There is a nonnegative $p$-superharmonic function $u(x)$ satisfying \eqref{Equ:p-laplace-measure}
and $\min\{u,\lambda\} \in W^{1,p}_0(\Omega)$ for any $\lambda >0$.
\end{lemma}

The following is also very important to us.

\begin{lemma}\label{Lem:cap-estimate} (\cite[Lemma 3.9]{KM94}) 
Suppose that $u\in W_0^{1,p}(\Omega)$ is $p$-superharmonic functions. Then for $\lambda>0$ it holds
\begin{equation}\label{Equ:cap-estimate}
\lambda^{p-1}cap_p(\{x\in\Omega:u(x)>\lambda\},\Omega)\le\mu(\Omega).
\end{equation}
\end{lemma}

To tackle the lack of continuity of $p$-superharmonic functions, the following $p$-thinness was introduced, for example, in \cite{HKM, KM94}.

\begin{definition}\label{Def:p-thin-wiener}  Let $E\subset\mathbb{R}^n$ and $x_0\in \mathbb{R}^n$. For $p\in (1, n]$, we say $E$ is $p$-thin 
for continuity at $x_0$ according to the Wiener criterion if
\begin{equation}\label{Equ:p-thin-wiener} 
W_p(E, x_0) = \int_0^1 \left(\frac {cap_p (B(x_0, t)\cap E, B(x_0, 2t))}{cap_p (B(x_0, t), B(x_0, 2t))}\right)^\frac 1{p-1} \frac {dt}t < \infty.
\end{equation}
\end{definition}

One seminal result in \cite{KM94} is the confirmation that the notion of $p$-thinness for continuity indeed satisfies the so-called Cartan property and therefore 
induces the fine topology for continuity of $p$-superharmonic functions. 

\begin{theorem}\label{Thm:km94-cartan} (\cite[Theorem 1.3]{KM94}) Let $p\in (1, n]$. And 
let $E\subset \mathbb{R}^n$ and $x_0\in \bar E\setminus E$. Then $E$
is $p$-thin for continuity according to the Wiener criterion in Definition \ref{Def:p-thin-wiener} if and only if there is an $\mathcal{A}$-superharmonic function
$u$ such that
$$
\lim_{x\to x_0 \text{ and } x \in E} u(x) > u(x_0) = \liminf_{x\to x_0} u(x).
$$
\end{theorem}


\subsection{$p$-thinness for the singular behavior} 

The notions of thinness in potential theory define so-called fine topologies for taking limits following Cartan's idea. 
The readers are referred to \cite{Br40, Br44, AM72, AH73, HeKi, HKM, KM94, MQ21} for detailed discussions and references therein. 
 To deal with the singular behavior of $p$-superharmonic functions, like \cite[Definition 3.1]{MQ21} (see also
\cite[Section 2.5]{Mi96}), we propose a notion of thinness that is different from that by Definition \ref{Def:p-thin-wiener} in general. Let 
\begin{eqnarray*}
\omega_i(x_0)&=&\{x\in\mathbb R^n:2^{-i-1}\le|x-x_0|\le2^{-i}\};\\
\Omega_i(x_0)&=&\{x\in\mathbb R^n:2^{-i-2}\le|x-x_0|\le2^{-i+1}\}.
\end{eqnarray*}

\begin{definition}\label{Def:quasi-p-thin}
For $p\in (1, n)$, a set $E\subset\mathbb R^n$ is said to be $p$-thin for the singular behavior at $x_0\in\mathbb R^n$ if 
\begin{equation}\label{Equ:quasi-p-thin}
\sum_{i=1}^\infty \frac{cap_p(E\cap\omega_i(x_0),\Omega_i(x_0))}{cap_p(\partial B(x_0, 2^{-i}), B(x_0, 2^{-i+1}))}  <+\infty.
\end{equation}
Meanwhile a set $E$ is said to be $n$-thin for the singular behavior at $x_0\in \mathbb{R}^n$ if
\begin{equation}\label{Equ:quasi-n-thin}
\sum_{i=1}^\infty i^{n-1}\, cap_n(E\cap\omega_i (x_0), \Omega_i(x_0)) < +\infty.
\end{equation} 
\end{definition}

By comparing \eqref{Equ:quasi-p-thin} and \eqref{Equ:quasi-n-thin}
with the Wiener integral criterion \eqref{Equ:p-thin-wiener} or the equivalent Wiener sum criterion 
\begin{equation}\label{Equ:old-p-thin}
\sum_{i=1}^\infty \left(\frac{cap_p(B(x_0, 2^{-i})\cap E, B(x_0, 2^{-i+1}))}{cap_p(B(x_0, 2^{-i}), B(x_0, 2^{-i+1}))}\right)^\frac 1 {p-1}<+\infty
\end{equation}
(cf. \cite[(12.9)]{HKM}), 
it is clear that a set $E\subset\Omega\subset\mathbb{R}^n$ is $p$-thin for the singular behavior at $x_0\in\Omega$ by Definition \ref{Def:quasi-p-thin}
if it is $p$-thin for continuity at $x_0$ by Definition \ref{Def:p-thin-wiener} when $p\in [2, n)$. The relation is not straightforward when $p=n\geq 3$.
In fact, analogous to \cite[Theorem 5.3 and Theorem 6.3 in Chapter 2]{Mi96} on the Riesz potentials and Log potentials, we have 
Theorem \ref{Thm:asymptotic-wolff-intr}, which is restated for readers' convenience as follows:

\begin{theorem}\label{Thm:wolff potential upper bdd-singualarity}
Let $\mu$ be a nonnegative finite Radon measure in $\Omega$ and $B(x_0, 3r_0)\subset \Omega$. 
Then there is a subset $E$ that is $p$-thin for the singular behavior at $x_0$ such that
\begin{equation*}
\lim_{x\to x_0 \text{ and } x\notin E } |x-x_0|^\frac {n-p}{p-1} W_{1,p}^{\mu}(x, \, r_0) = \frac{p-1}{n-p} \mu(\{x_0\})^\frac 1{p-1}
\end{equation*}
for $p\in (1, n)$. Similarly, there is a subset $E$ that is $n$-thin for the singular behavior at $x_0$ such that
\begin{equation*}
\lim_{x\to x_0 \text{ and } x\notin E} \frac{W_{1,n}^{\mu}(x, \, r_0)}{\log\frac 1{|x-x_0|}} = \mu(\{x_0\})^\frac 1{n-1}.
\end{equation*}
\end{theorem}

\begin{proof} For convenience, let $x_0=0$. For $x\in \omega_i$, let us write
\begin{equation}\label{Equ:4-terms-sin}
\aligned
W_{1,p}^{\mu}& (x,r_{0}) =\int_{0}^{r_{0}}\left( \frac{\mu(B(x,t))}{t^{n-p}}\right)^{\frac{1}{p-1}}\frac{dt}{t}\\
 & = \left\{ \int_{2^{i_0}|x|}^{r_0}  +  \int_{(1+\delta)|x|}^{2^{i_0}|x|}  + \int_{(1-\delta)|x|}^{(1+\delta)|x|} + \int_{0}^{(1-\delta)|x|}
 \right\}
\left( \frac{\mu(B(x,t))}{t^{n-p}}\right)^{\frac{1}{p-1}}\frac{dt}{t}
\endaligned
\end{equation}
for some $i_0$ large and $\delta > 0$ small to be fixed. For the first term in the right side of \eqref{Equ:4-terms-sin}
\begin{align*}
\int_{2^{i_0}|x|}^{r_0}\left( \frac{\mu(B(x,t))}{t^{n-p}}\right)^{\frac{1}{p-1}}\frac{dt}{t} \leq \frac {p-1}{n-p}\mu(\Omega)^\frac 1{p-1} (2^{i_0}|x|)^{-\frac {n-p}{p-1}}.
\end{align*}
Clearly, this is the easiest term and
$$
\lim_{x\to 0} |x|^\frac {n-p}{p-1} \int_{2^{-i_0 +1}}^{r_0} \left( \frac{\mu(B(x,t))}{t^{n-p}}\right)^{\frac{1}{p-1}}\frac{dt}{t} \leq C2^{-i_0\frac {n-p}{p-1}} \to 0
$$
as $i_0\to \infty$. For the second term in the right side of \eqref{Equ:4-terms-sin}
$$
\int_{(1+\delta)|x|}^{2^{i_0}|x|}  \left( \frac{\mu(B(x,t))}{t^{n-p}}\right)^{\frac{1}{p-1}} \frac{dt}{t} 
 =   (\mu(B(x, t_{i, i_0}))^\frac 1{p-1} \frac {p-1}{n-p}(((1+\delta)|x|)^{ - \frac {n-p}{p-1} } - (2^{i_0}|x|)^{-\frac {n-p}{p-1}})
$$
for some $t_{i, i_0} \in ((1+\delta)|x|, 2^{i_0}|x|)$. Hence
$$
\aligned
\lim_{x\to 0} |x|^\frac {n-p}{p-1}  & \int_{(1+\delta)|x|}^{2^{i_0|x|}}   \left( \frac{\mu(B(x,t))}{t^{n-p}}\right)^{\frac{1}{p-1}} \frac{dt}{t} 
= \frac {p-1}{n-p} \mu(\{0\})^\frac 1{p-1} (\frac 1{(1+\delta)^\frac {n-p}{p-1}} - \frac 1{2^{i_0\frac {n-p}{p-1}}})\\
& \to \frac {p-1}{n-p}\mu(\{0\})^\frac 1{p-1}
\endaligned
$$
as $i_0\to \infty$ and $\delta\to 0$. 
For the third term in the right side of \eqref{Equ:4-terms-sin}
$$
\int_{(1-\delta)|x|}^{(1+\delta)|x|} \left( \frac{\mu(B(x,t))}{t^{n-p}}\right)^{\frac{1}{p-1}}\frac{dt}{t} 
\leq \frac {p-1}{n-p} \mu(\Omega)^\frac 1{p-1} (((1-\delta)|x|)^{-\frac {n-p}{p-1}} - ((1+\delta)|x|)^{-\frac {n-p}{p-1}}).
$$
Therefore
$$
\lim_{x\to 0}|x|^\frac{n-p}{p-1}\int_{(1-\delta)|x|}^{(1+\delta)|x|}  \left( \frac{\mu(B(x,t))}{t^{n-p}}\right)^{\frac{1}{p-1}}\frac{dt}{t} 
\leq \frac {p-1}{n-p} \mu(\Omega)^\frac 1{p-1} ((1-\delta)^{-\frac {n-p}{p-1}} - (1+\delta)^{-\frac {n-p}{p-1}}) \to 0
$$
as $\delta\to 0$. For the fourth term in the right side of \eqref{Equ:4-terms-sin}, we now switch the gear and estimate $cap_p(E_i, \Omega_i)$, 
where
\begin{equation}\label{Equ:wolff-for-small-r}
E_i = \{x\in \omega_i: W_{1, p}^\mu(x, (1-\delta)|x|) \geq \alpha_i^{-\frac 1{p-1}} |x|^{-\frac{n-p}{p-1}}\}
\end{equation}
and $\{\alpha_i\}$ is chosen so that
$$
\lim_{i\to \infty} \alpha_i= \infty \text{ but } \sum_i \alpha_i\mu(\Omega_i) < \infty
$$
in the light of  \cite{Rey} (cf. \cite[(5) Page 40]{Brom})  and $\sum_i \mu(\Omega_i) < \infty$.
To apply Lemma \ref{Lem:cap-estimate}, we let $u_i$ be the nonnegative $p$-superharmonic function that solves 
$$
\left\{\aligned -\Delta_p u_i & = \mu \text{ in $\Omega_i$}\\
u_i & = 0 \text{ on $\partial \Omega_i$} \endaligned\right..
$$
The existence of such $u_i$ is given by Lemma \ref{Lem:existence-p-sh} 
(cf. \cite[Theorem 2.4]{KM92}). Applying Theorem \ref{Thm:main-use-wolff} (cf. \cite[Theorem 1.6]{KM94}), we have
$$
c(n, p)W_{1, p}^\mu (x, (1-\delta)|x|) \leq u_i(x) \text{ for $x\in \omega_i$.}
$$
Therefore, in the light of Lemma \ref{Lem:cap-estimate} (cf. \cite[Lemma 3.9]{KM94}), we have
$$
cap_p(E_i, \Omega_i) \leq c(n, p) \mu(\Omega_i)\alpha_i |x|^{n-p} \leq c(n, p) 2^{-i(n-p)} \alpha_i\mu(\Omega_i).
$$ 
Thus, if let $E = \underset{i}{\bigcup} E_i \), we have
$$
\sum_i 2^{(n-p)i} cap_p (E\cap\omega_i, \Omega_i) \leq c(n, p) \sum_i \alpha_i\mu(\Omega_i) < \infty,
$$
which implies that $E$ is $p$-thin according to Definition \ref{Def:quasi-p-thin}. Here we have used the facts 
that are stated as (2) and (4) in Lemma \ref{Lem:subadditivity-Lip}. So the proof is completed for $p\in (1, n)$.
\\

For $p=n$, the proof goes similarly. We also start with
\begin{equation}\label{Equ:4-terms-sin-n}
\aligned
W_{1,n}^{\mu}& (x,r_{0}) =\int_{0}^{r_{0}} \mu(B(x,t))^{\frac{1}{n-1}} \frac{dt}{t}\\
 & = \left\{ \int_{(1+\delta)|x|^\delta}^{r_0}  +  \int_{(1+\delta)|x|}^{(1+\delta)|x|^\delta}  + \int_{(1-\delta)|x|}^{(1+\delta)|x|} + \int_{0}^{(1-\delta)|x|}
 \right\} \mu(B(x,t))^{\frac{1}{n-1}} \frac{dt}{t}
\endaligned
\end{equation}
for some $\delta > 0$ small to be fixed. For the first term in the right side of \eqref{Equ:4-terms-sin-n}
\begin{align*}
\int_{(1+\delta)|x|^\delta}^{r_0} \mu(B(x,t))^{\frac{1}{n-1}} \frac{dt}{t} \leq \mu(\Omega)^\frac 1{n-1} (\log \frac{r_0}{1+\delta} + \delta\log \frac 1{|x|}).
\end{align*}
This is again an easy term, since
$$
\lim_{x\to 0} \frac 1{-\log|x|} \int_{(1+\delta)|x|^\delta}^{r_0} \mu(B(x,t))^{\frac{1}{n-1}} \frac{dt}{t} \to 0
$$
as $\delta\to 0$. For the second term in the right side of \eqref{Equ:4-terms-sin-n}
$$
\int_{(1+\delta)|x|}^{(1+\delta)|x|^\delta} \mu(B(x,t))^{\frac{1}{n-1}} \frac{dt}{t}
 =   (\mu(B(x, t_\delta))^\frac 1{n-1} (1-\delta) \log\frac 1{|x|}
$$
for some $t_\delta \in ((1+\delta)|x|, (1+\delta)|x|^\delta)$. Hence
$$
\lim_{x\to 0} \frac 1{-\log|x|}  \int_{(1+\delta)|x|}^{(1+\delta)|x|^\delta} \mu(B(x,t))^{\frac{1}{n-1}} \frac{dt}{t}
\to \mu(\{0\})^\frac 1{n-1}
$$
as $\delta\to 0$. Because
$$
\bigcap B(x, t_\delta) = \{0\} \text{ as $x\to 0$}.
$$
For the third term in the right side of \eqref{Equ:4-terms-sin-n}
$$
\int_{(1-\delta)|x|}^{(1+\delta)|x|} \mu(B(x,t))^{\frac{1}{n-1}} \frac{dt}{t}
\leq\mu(\Omega)^\frac 1{n-1} \log\frac{1+\delta}{1-\delta}.
$$
This turns out to the easiest term
$$
\lim_{x\to 0} \frac 1{-\log |x|} \int_{(1-\delta)|x|}^{(1+\delta)|x|}  \mu(B(x,t))^{\frac{1}{n-1}} \frac{dt}{t} \to 0.
$$
For the fourth term in the right side of \eqref{Equ:4-terms-sin-n}, we again switch the gear and estimate $cap_p(E_i, \Omega_i)$, 
where
\begin{equation}\label{Equ:wolff-for-small-r}
E_i = \{x\in \omega_i: W_{1, n}^\mu(x, (1-\delta)|x|) \geq \gamma_i^{-\frac 1{n-1}} \log\frac 1{|x|} \}
\end{equation}
and $\{\gamma_i\}$ is chosen so that
$$
\lim_{i\to \infty} \gamma_i= \infty \text{ but } \sum_i \gamma_i\mu(\Omega_i) < \infty
$$
in the light of  \cite{Rey} (cf. \cite[(5) Page 40]{Brom}) and $\sum_i \mu(\Omega_i) < \infty$.
Again, we let $u_i$ be the nonnegative $n$-superharmonic function that solves 
$$
\left\{\aligned -\Delta_n u_i & = \mu \text{ in $\Omega_i$}\\
u_i & = 0 \text{ on $\partial \Omega_i$} \endaligned\right..
$$
And we apply Theorem \ref{Thm:main-use-wolff} (cf. \cite[Theorem 1.6]{KM94}) to have
$$
c(n)W_{1, n}^\mu (x, (1-\delta)|x|) \leq u_i(x) \text{ for $x\in \omega_i$.}
$$
and
$$
cap_n(E_i, \Omega_i) \leq c(n) \mu(\Omega_i) i^{-(n-1)}\gamma_i.
$$ 
Thus, if let $E = \underset{i}{\bigcup} E_i$, we have
$$
\sum_i i^{n-1} cap_p (E\cap\omega_i, \Omega_i) \leq c(n) \sum_i \gamma_i\mu(\Omega_i) < \infty,
$$
which implies that $E$ is $n$-thin according to Definition \ref{Def:quasi-p-thin}. 
\end{proof}

\subsection{The $p$-fine topology}\label{Subsec:fine-topology}
In potential theory, the introduction of a fine topology based on a notion of thin sets often provides insightful analysis of 
the subtle continuity deficiency under usual Euclidean topology. 
For the Riesz potentials, in \cite[Theorem 5.1 and 5.2 in Section 2.5]{Mi96}, it has been verified that the topology based on the notion of 
$\alpha$-thinness defined at the beginning of \cite[Section 2.5]{Mi96} is the fine one for $\alpha$-continuity. Interestingly, the same fine topology 
also works for the $\alpha$-limit of the Riesz potentials divided by the kernels at singular points 
shown in \cite[Theorem 5.3 and 5.4 in Section 2.5]{Mi96}, which gives the 
effective description of the asymptotic behavior and has had geometric applications in our
recent work \cite{MQ21,MQ021,MQ22, MQ23}. Similarly, in dimension 2, as pointed out in the 
discussion near the end of \cite{AH73} (cf. \cite[Section 2.2]{MQ21}), 
the two different notions of thinness for continuity and the singular behavior for the Log potentials 
turn out to be equivalent due to \cite[Theorem 2]{Br44} and therefore induce the same fine topology.
\\

For $p$-superharmonic functions, the $p$-thinness for continuity by the Wiener integral criterion \eqref{Equ:p-thin-wiener} in 
Definition \ref{Def:p-thin-wiener} has been shown to be good for a fine topology in \cite[Theorem 1.1 and Theorem 1.3]{KM94} for $p\in (1, n]$, 
analogous to \cite[Theorem 5.1 and 5.2 in Section 2.5]{Mi96} for the Riesz potentials. 
However, contrary to what are described in the above about the Riesz potentials and the Log potentials, it is remarkable that
the notion of $p$-thinness in Definition \ref{Def:quasi-p-thin} for the singular behavior differs from that by Definition \ref{Def:p-thin-wiener} for 
continuity when $p\in (1, n)$ or $p=n\geq 3$. 
Theorem \ref{Thm:wolff potential upper bdd-singualarity} as the analogue of \cite[Theorem 5.3 and Theorem 6.3 in Chapter 2]{Mi96}
on the Riesz and Log potentials in the previous subsection tells that the topology induced by $p$-thinness for the singular behavior indeed is sufficient 
to give the asymptotics. Next we want to establish the analogue of \cite[Theorem 5.4 in Section 2.5]{Mi96} to confirm the topology induced by the 
$p$-thinness for the singular behavior in Definition \ref{Def:quasi-p-thin} is indeed the fine topology for getting asymptotics for $p\in(1, n]$. 
Namely, we have Theorem \ref{Thm:fine-topology-intr}, which is restated for readers' convenience as follows:

\begin{theorem}\label{Thm:converse-p-thin} Suppose that $E$ is a subset that is $p$-thin for the singular behavior at the origin
according to Definition \ref{Def:quasi-p-thin} for $p\in (1, n]$. And suppose that the origin is in $\bar E\setminus E$. Then, when $p\in (1, n)$, 
there is a Radon measure $\mu$ in a neighborhood of the origin such that, for some fixed $r_0>0$, 
$$
\lim_{x\to 0 \text{ and } x\in E} |x|^\frac {n-p}{p-1}W^\mu_{1, p}(x, r_0) = \infty.
$$
Similarly, when $p=n$, there is a Radon measure $\mu$ in a neighborhood of the origin such that, for some fixed $r_0>0$, 
$$
\lim_{x\to 0 \text{ and } x\in E} \frac {W^\mu_{1, n}(x, r_0)}{\log\frac 1{|x|}} = \infty.
$$
\end{theorem}

\begin{proof} Let us prove the theorem for $p\in (1, n)$ first. To start, from (4) in Lemma \ref{Lem:subadditivity-Lip}, we may rewrite \eqref{Equ:quasi-p-thin} as 
$$
\sum_i 2^{i(n-p)} cap_p(E\cap\omega_i, \Omega_i) < \infty.
$$
Secondly, by the definition, it is easily seen that
$$
 cap_p(E\cap\omega_i, B(0, 2^{-i+3}))\leq cap_p(E\cap\omega_i, \Omega_i),
$$
which implies
$$
\sum_i 2^{i(n-p)} cap_p(E\cap\omega_i, B(0, 2^{-i+3})) < \infty.
$$
Next, again by the definition, for each $i$, we pick up open sets $\hat E_i$ such that
$$
E\cap \omega_i \subset \hat E_i \subset \Omega_i \subset B(0, 2^{-i+3})
$$
and
$$
\sum_i 2^{i(n-p)} cap_p(\hat E_i, B(0, 2^{-i+3})) < \infty.
$$
In the light of \cite{Rey} (cf. \cite[(5) Page 40]{Brom}) , there is a sequence of number $\{\kappa_i\}$ such that $\kappa_i\to\infty$ and
\begin{equation}\label{Equ:kappa-p-thin}
\sum_i 2^{i(n-p)} \kappa_i cap_p(\hat E_i, B(0, 2^{-i+3})) < \infty.
\end{equation}
Now, based on Theorem \ref{Thm:capacitary-function} (cf. \cite[Theorem 9.35]{HKM}), we let $\hat R_{\hat E_i}$ be the capacitary function
of $\hat E_i$ in $B(0, 2^{-i+3})$ and 
$$
\mu_i = -\Delta_p \hat R_{\hat E_i}
$$
be the capacitary distribution of $\hat E_i$ in $B(0, 2^{-i+3})$.  We now set
$$
\mu = m\delta_0 + \sum_i 2^{i(n-p)} \kappa_i \hat\mu_i,
$$
for any $m\geq 0$, where $\delta_0$ is the Dirac measure at the origin and $\hat\mu_i (A) = \mu_i(A\cap B(0, 2^{-i +3}))$. 
And We claim $\mu$ is the Radon measure on $B(0, 2)$ that we want. 
First, in the light of Lemma \ref{Lem:dual-capacity}, we have
$$
\hat \mu_i (B(0, 2)) =\mu_i(B(0, 2^{-i+3})) = cap_p(\hat E_i, B(0, 2^{-i+3})),
$$
which implies $\mu$ is a finite Radon measure. And, since each $\hat E_i$ is open, we know, for all 
$x\in E\cap\omega_i\subset \hat E_i$, 
\begin{equation}\label{Equ:lower-bd-key}
1 = \hat R_{\hat E_i} (x) \leq c_2 (\inf_{B(x, 2^{-i+1})} \hat R_{\hat E_i} + W^{\mu_i}_{1, p}(x, \cdot 2^{-i+2})) 
\end{equation}
due to Theorem \ref{Thm:main-use-wolff} (\cite[Theorem 1.6]{KM94}). To handle $\inf_{B(x, 2^{-i+1})}\hat R_{\hat E_i}$, we apply Lemma 
\ref{Lem:cap-estimate} (\cite[Lemma 3.9]{KM94}) and get
$$
\aligned
cap_p(\{x: \hat R_{\hat E_i} (x) & > \frac 12 \kappa_i^{-\frac 1{p-1}}\}, B(0, 2^{-i+3}) \leq  2^{p-1}\kappa_i
cap_p(\hat E_i, B(0, 2^{-i+3}))\\
& \leq \varepsilon cap_p(B(0, 2^{-i}), B(0, 2^{-i+3}))
\endaligned
$$
for some small number $\varepsilon >0$, because of \eqref{Equ:kappa-p-thin}. Therefore
$$
\inf_{B(x, 2^{-i+1})}\hat R_{\hat E_i} \leq \inf_{B(0, 2^{-i})}\hat R_{\hat E_i} \leq \frac 12\kappa_i^{-\frac 1{p-1}} \to 0, 
$$
because $B(0, 2^{-i})\subset B(x, 2^{-i+1})$, which implies
$$
W^{\mu_i}_{1, p} (x, 2^{-i+2}) \geq \frac 1{2c_2}
$$
by \eqref{Equ:lower-bd-key}. Thus, for $x\in E\cap\omega_i \subset \hat E_i$, 
$$
W^\mu_{1, p}(x, \frac 14) \geq W^{2^{i(n-p)}\kappa_i\mu_i}_{1, p} (x, 2^{-i+2}) \geq c |x|^{-\frac {n-p}{p-1}} \kappa_i^\frac 1{p-1} \frac 1{2c_2}
$$
and
$$
|x|^\frac {n-p}{p-1} W^\mu_{1,p} (x, \frac 14) \geq c \kappa_i^\frac 1{p-1} \to \infty
$$
as $i\to \infty$. This completes the proof of the theorem for $p\in (1, n)$. 
\\

For $p=n$, since the proof goes similarly, we only give a sketch. 
We may start with open set $\hat E_i$ and $\tau_i\to\infty$ such that 
$$
E\cap\omega_i \subset \hat E_i \subset \Omega_i \subset B(0, 2^{-i+3}) 
$$ 
and
$$
\sum_i i^{n-1} \tau_i cap_p(\hat E_i, B(0, 2^{-i+3})) < \infty.
$$
And then let 
$$
\mu = m\delta_0 + \sum_i i^{n-1} \tau_i \hat\mu_i
$$
for any $m\geq 0$, where $\delta_0$ is the Dirac measure at the origin, $\hat\mu_i (A) = \mu_i(A\cap B(0, 2^{-i +3}))$, and 
$\mu_i$ is the capacitary distribution of $\hat E_i$ in $B(0, 2^{-i+3})$.
Similarly, we want to 
show $\mu$ is the Radon measure on $B(0, 2)$ that we want. For  $x\in E\cap\omega_i\subset\hat E_i$, we have \eqref{Equ:lower-bd-key}
and 
$$
\inf_{B(x, 2^{-i+1})} \hat R_{\hat E_i} \leq \frac 12\tau_i^{-\frac 1{n-1}} \to 0,
$$
which implies
$$
W^{\mu_i}_{1, n} (x, 2^{-i+2}) \geq \frac 1{2c_2}.
$$
Thus, for $x\in E\cap\omega_i\subset \hat E_i$, 
$$
W^\mu_{1, n} (x, \frac 14) \geq W^{i^{n-1}\tau_i\mu_i}_{1, n} (x, 2^{-i+2}) \geq c (\log \frac 1{|x|}) \tau_i^{\frac 1{n-1}} \frac 1{2c_2}
$$
and
$$
\frac {W^\mu_{1, n}(x, \frac 14)}{\log\frac 1{|x|}} \geq c \tau_i^\frac 1{n-1} \to \infty
$$
as $i\to\infty$. This completes the proof of the theorem for $p=n$.
\end{proof}

\section{On $p$-superharmonic functions and applications}\label{Sec:asymptotics-p-superharmonic}

In this section we discuss the asymptotic behavior of $p$-superharmonic functions at singularities based on that of the Wolff potentials in
Theorem \ref{Thm:wolff potential upper bdd-singualarity}. We also include some corollaries for solutions to fully nonlinear equations. 

\subsection{Arsove-Huber type theorem for $p$-superharmonic functions}\label{Subsec:arsove-huber}
For the convenience of readers, let us restate Theorem \ref{Thm:main-asymptotic-intr} as follows:

\begin{theorem}\label{Thm:main-asymptotic} Suppose that $u$ is a nonnagetive $p$-superharmonic function in $\Omega\subset\mathbb{R}^n$ satisfying
$$
-\Delta_p u = \mu
$$
for a nonnegative finite Radon measure $\mu$ on $\Omega$ and $p\in (1, n]$. Then, for $x_0\in\Omega$, there is a subset $E$ that is $p$-thin for singular
behavior at $x_0$ such that
\begin{equation}\label{Equ:p-super-asymptotic}
\lim_{x\to x_0\text{ and } x\notin E}  \frac{u(x)}{G_p(x, x_0)} = m
\end{equation} 
where $\alpha_0(n, p) m^{p-1} = \mu(\{0\})$ and $\alpha_0(n, p)$ is from \eqref{Equ:p-green}. Moreover
\begin{equation}\label{Equ:lower-bdd} 
u(x) \geq m|x|^{- \frac {n-p}{p-1}} - c_0
\end{equation}
for some $c_0$ and all $x$ in a neighborhood of $x_0$, where $G_p(x, x_0)$ is defined in \eqref{Equ:G-p}.
\end{theorem}

The asymptotic behavior of $p$-superharmonic functions in \eqref{Equ:p-super-asymptotic} extends those for superharmoic in dimension 2 in \cite{AH73} 
(cf. \cite{MQ21, MQ021} for $n$-superharmonic functions in higher dimensions) and $p$-harmonic functions 
in \cite{KV} at singularities. Our proof is focsued on $p\in (1, n)$ and consists of five steps. In Step 1, we will derive
consequences of the weak comparison principle and draw our attention to 
\begin{equation}\label{Equ:m=liminf}
m = \liminf_{x\to 0} \frac {u(x)}{|x|^{-\frac {n-p}{p-1}}}.
\end{equation}
In Step 2, we recall that, any sequential weak limit of the rescaling is of the form $\lambda|\xi|^{-\frac {n-p}{p-1}}$ for some nonnegative $\lambda$ based
on the uniqueness \cite[Theorem 2.2]{KV}. Then in Step 3, we push forward and 
show that all possible
sequential limits are the same $m|\xi|^{-\frac {n-p}{p-1}}$, where $m$ is given in \eqref{Equ:m=liminf}. Next, in Step 4, we show that
\begin{equation}\label{Equ:p-thin-limit}
\lim_{x\to 0 \text{ and } x\notin E} \frac {u(x)}{|x|^{-\frac {n-p}{p-1}}} = m
\end{equation}
for some subset $E$ that is $p$-thin for the singular behavior at $x_0$ based on Theorem \ref{Thm:wolff potential upper bdd-singualarity}. Finally, in Step 5, 
we calculate what is $m$ in terms of the point charge $\mu(\{0\})$.

\begin{proof}[Proof of Theorem \ref{Thm:main-asymptotic}] The proof for $p=n$ goes similarly in principle, and has been treated, not completely the same though, 
in \cite{MQ21, MQ021}. In particular, Step 5 in the following for $p=n$ seems more explicit than that in \cite{MQ21, MQ021}. Therefore, we will focus on $p\in (1, n)$, 
even though some of the steps (including Step 5) actually work for $p=n$ with no or little changes. We will let $x_0=0$ for simplicity. 

\vskip 0.1in\noindent{\bf Step 1.} \quad First, let us derive consequences including \eqref{Equ:lower-bdd} from the weak comparison principle 
for $p$-Laplace equations (cf. Definition \ref{Def:p-superhar}). Recall
$$
G_p(x, 0) =\left\{\aligned |x|^{-\frac {n-p}{p-1}} & \text{ when } p\in (1, n)\\
- \log |x| & \text{ when } p=n\endaligned\right..
$$
It is easily calculated that
\begin{equation}\label{Equ:p-green}
-\Delta_p G_p(x,0) =\left\{\aligned  (\frac{n-p}{p-1})^{p-1}|\mathbb{S}^{n-1}|\delta_0,  & \quad p\in (1, n)\\
|\mathbb{S}^{n-1}|\delta_0, & \quad p=n\endaligned\right.
= \alpha_0(n, p)\delta_0 \text{ in $\mathbb{R}^n$},
\end{equation}
where $\delta_0$ is the Dirac function at the origin. Obviously $G_p(x, 0)$ is $p$-harmonic away from the origin. 

Notice that, as in \cite{KM94}, the infimum here really means the essential infimum for functions, in fact, in many literatures the notation 
$\text{ess}\liminf$ is used instead. 
To have better idea what $m$ can possibly be, let us consider
\begin{equation}\label{Equ:m(r)}
m(r) = \inf_{|x| = r} \frac {u(x)}{|x|^{-\frac {n-p}{p-1}} - \gamma_0} = \frac {\inf_{|x|=r} u(x)}{r^{-\frac {n-p}{p-1}} - \gamma_0},
\end{equation}
for $\gamma_0 = |3r_0|^{-\frac {n-p}{p-1}}$ and $r\in (0, 3r_0)$. $m(r)$ turns out to be monotonically non-increasing as $r\to 0$ from
the weak comparison principle, similar to \cite[Lemma 3.6]{MQ21} (cf. \cite[Lemma 3.1]{Tol83} and \cite{KV}). 
Hence
\begin{equation}\label{Equ:m-infty=m}
\lim_{r\to 0} m(r) = \liminf_{x\to 0} \frac {u(x)}{|x|^{-\frac {n-p}{p-1}} - \gamma_0} =  \liminf_{x\to 0} \frac {u(x)}{|x|^{-\frac {n-p}{p-1}}} = m\in [0, \infty).
\end{equation}
To establish \eqref{Equ:lower-bdd}, we start with a constant $c_0$ for any $\epsilon >0$ such that  
$$
u(x) \geq (m -\epsilon)|x|^{-\frac {n-p}{p-1}} -c_0 \text{ on $\partial B(0, 3r_0)$}.
$$
Moreover we know, on $\partial B(0, s)$ for any 
$s\in (0, 3r_0)$,
$$
u(x) \geq (m-\epsilon) |x|^{-\frac{n-p}{p-1}} - c_0
$$
at least when $s$ is sufficiently small by the definition of $m$ in \eqref{Equ:m-infty=m}. Therefore we have
$$
u(x) \geq (m-\epsilon)|x|^{-\frac {n-p}{p-1}} - c_0 
$$
for all $x\in B(0, 3r_0)\setminus \{0\}$, which implies \eqref{Equ:lower-bdd} since $\epsilon$ is arbitrarily small.

\vskip 0.1in\noindent{\bf Step 2.} \quad
To perform the rescaling approach as in \cite{KV, MQ21}, we consider  
$$
u_r (\xi) = \frac {u(r\xi)}{r^{-\frac {n-p}{p-1}}} \text{ as $r\to 0$}.
$$
First, due to Theorem \ref{Thm:wolff potential upper bdd-singualarity} and Theorem 
\ref{Thm:main-use-wolff}, there is a subset $E^\mu$ that is $p$-thin for singular behavior and a constant $\hat c$
such that 
\begin{equation}\label{Equ:C-hat}
0 < |x|^\frac{n-p}{p-1} u(x) \leq \hat c
\end{equation}
for all $x\in (B(0, r_0)\setminus E^\mu) \setminus\{0\}\subset \Omega\setminus\{0\}$. Hence, for the sake of the proof, we turn to 
\begin{equation}\label{Equ:u-hat}
\hat u (x) = \min \{ u(x), (\hat c+1)|x|^{-\frac {n-p}{p-1}}\},
\end{equation}
which still is a nonnegative $p$-superharmonic function and therefore
\begin{equation}
-\Delta_p \hat u = \hat \mu \text{ in $\Omega$}
\end{equation}
for some nonnegative finite Radon measure $\hat\mu$ on $\Omega$. Notice that $u(x) \equiv \hat u(x)$ at least away from the subset $E^\mu$.  
Now, for $r>0$, we consider the rescaling as in \cite{KV} (see also \cite{MQ21})
$$
\hat u_r (\xi) = \frac {\hat u(r\xi)}{r^{-\frac {n-p}{p-1}}} \text{ and } - \Delta_p^\xi \hat u_r = \hat \mu_r \text{ in $A_{0, \frac {3r_0}r}$}
$$
where
$$
A_{0, \frac {3r_0}r} = \{\xi\in \mathbb{R}^n: |\xi|\in (0, \frac {3r_0}r)\} \text{ and } d\hat\mu_r (\xi) = r^n d\hat \mu (r\xi).
$$
By the definition of $\hat u (x)$, we now know that, for $r>0$,
\begin{equation}\label{Equ:bdd-uniqueness}
|\xi|^\frac {n-p}{p-1} \hat u_r(\xi) = \frac {\hat u(r\xi)}{ (r|\xi|)^{-\frac{n-p}{p-1}}} \leq \hat c +1 \text{ for all $\xi\in A_{0, \frac{3r_0}r}$,}
\end{equation}
which is important when applying \cite[Theorem 2.2]{KV} to identify $\hat u^\varsigma_\infty = m^\varsigma|\xi|^{-\frac {n-p}{p-1}}$ in 
the following Lemma \ref{Lem:known-to-veron}. 
Meanwhile, consequently, $\hat u_r(\xi)$ is bounded in $A_{s, \frac 1s}$ for any given $s\in (0, 1)$. Moreover, for $s\in (0, 1)$, 
$$
\hat\mu_r(A_{s, \frac 1s}) = \int_{A_{s, \frac 1s}} d\hat\mu_r (\xi) = \int_{B(0, \frac rs)\setminus B(0, rs)} d\hat\mu (x)
\to 0 \text{ as $r\to 0$.}
$$
Now we are ready to state a lemma based on \cite[Theorem 2.2]{KV} and \cite[Theorem 4.3.8]{Vn17}. 

\begin{lemma}\label{Lem:known-to-veron} (\cite{KV, Vn17}) Suppose that, for a sequence $\varsigma = \{r_k\}$ going to $0$, the rescaled sequence
\begin{equation}\label{Equ:w^{1,p}-weak}
\hat u_{r_k}\rightharpoonup \hat u^{\varsigma}_\infty \text{ weakly in $W^{1, p}_{\text{loc}}(\mathbb{R}^n\setminus \{0\})$ as $k\to\infty$}.
\end{equation}
Then, for a subsequence $\{r_k'\}\subset\{r_k\}$, we have
\begin{equation}\label{Equ:w^{1,p}-strong}
\hat u_{r_k'}\to \hat u^{\varsigma}_\infty \text{ strongly in $W^{1, p}_{\text{loc}}(\mathbb{R}^n\setminus \{0\})$ as $k\to\infty$}
\end{equation}
and
\begin{equation}\label{Equ:veron-uniqueness}
\hat u^{\varsigma}_\infty (\xi) = m^{\varsigma} |\xi|^{-\frac {n-p}{p-1}} \text{ in $\mathbb{R}^n$}
\end{equation}
for a nonnegative number $m^{\varsigma}$.
\end{lemma}
For convenience, from now on in our presentation, we will not differentiate a subsequence from the sequence with no loss of generality.

\vskip 0.1in\noindent{\bf Step 3.} \quad In this step, we show that $m^\varsigma = m$ for any sequence $\varsigma$.
From \eqref{Equ:m-infty=m} we know
\begin{equation}\label{Equ:m-inf=m-infty=m}
\lim_{k\to\infty} (\inf_{|\zeta|=|\xi|} \frac{\hat u_{r_k} (\zeta)}{|\zeta|^{-\frac{n-p}{p-1}}})= m \text{ and } 
\lim_{k\to\infty} (\inf_{A_{s, \frac 1s}} \frac{\hat u_{r_k} (\zeta)}{|\zeta|^{-\frac{n-p}{p-1}}})= m
\end{equation}
for any $\xi\neq 0$ and $s\in (0, 1)$. We are now ready to prove the claim that 
$m^\varsigma = m$ for any sequence $\varsigma = \{r_1, r_2, \cdots\}$ going to $0$. Should $\hat u_{r_k}$ be known to converges
to $\hat u^\varsigma_\infty$ pointwisely, this claim would have been completely trivial from \eqref{Equ:m-inf=m-infty=m}. Hence we want to explore 
some convergence following from \eqref{Equ:w^{1,p}-strong}.  From \eqref{Equ:m-inf=m-infty=m} and
$$
\hat u_r (\xi) \geq  (\inf_{|\zeta|=|\xi|} \frac {\hat u_r (\zeta)}{|\zeta|^{-\frac {n-p}{p-1}}})\, |\xi|^{-\frac {n-p}{p-1}} 
$$
we have, for a fixed $s\in (0, 1)$, 
$$
0 \leq \int_{A_{s, \frac 1s}} (\hat u_{r_k} (\xi) -  ((\inf_{|\zeta|=|\xi|} \frac {\hat u_{r_k}(\zeta)}{|\zeta|^{-\frac {n-p}{p-1}}})\, |\xi|^{-\frac {n-p}{p-1}} ) d\xi
\to (m^\varsigma - m)\int_{A_{s, \frac 1s}} |\xi|^{-\frac{n-p}{p-1}} d\xi 
$$
as $k\to\infty$, which implies $m^\varsigma \geq m$. Here we use the fact that 
$$
\int_{A_{s, \frac 1s}} \hat u_{r_k}(\xi) d\xi \to \int_{A_{s, \frac 1s}} m^\varsigma |\xi|^{-\frac {n-p}{p-1}} d\xi
$$
as $k\to\infty$ from \eqref{Equ:w^{1,p}-weak} and the consequence from Lemma \ref{Lem:known-to-veron}. 
On the other hand, by the lower semi-continuity, there is $\xi_k$ with $|\xi_k|=1$ such that
$$
\inf_{|\zeta|=1} \hat u_{r_k} (\zeta) = \hat u_{r_k}(\xi_k).
$$
Then, by Theorem \ref{Thm:main-use-wolff}, for $\xi\in B(\xi_k, \frac 12 \rho)$ for some small and fixed $\rho >0$,
$$
\aligned
\hat u_{r_k} (\xi) - & \inf_{B(\xi_k, 3\rho)}\hat u_{r_k}  \leq c_2(\inf_{B(\xi, \rho)}(\hat u_{r_k} - \inf_{B(\xi_k, 3\rho)}\hat u_{r_k}) 
+ W^{\hat\mu_{r_k}}_{1, p}(\xi, 2\rho)) \\
& \leq c_2(\inf_{B(\xi, \rho)}\hat u_{r_k} - \inf_{B(\xi_k. 3\rho)}\hat u_{r_k} + W^{\hat\mu_{r_k}}_{1, p}(\xi, 2\rho))
\endaligned
$$ 
in the light of \eqref{Equ:wolff-for-small-r} in the proof of Theorem \ref{Thm:wolff potential upper bdd-singualarity}. Notice that
$$
\inf_{B(\xi, \rho)}\hat u_{r_k},  \inf_{B(\xi_k. 3\rho)}\hat u_{r_k}\in [\inf_{A_{1-3\rho, 1+3\rho}}\hat u_{r_k}, \
\hat u_{r_k}(\xi_k)].
$$
and therefore the difference goes to zero as $k\to \infty$ and $\rho\to 0$. 

Let $\xi_k \to \xi_\infty\in \{\xi\in \mathbb{R}^n: |\xi|=1\}$ for simplicity, otherwise one may consider for a subsequence. 
Then, at least for $k$ sufficiently large,
$$
\hat u_{r_k}(\xi) \leq \hat u_{r_k}(\xi_k) + W^{\hat\mu_{r_k}}_{1, p} (\xi, 2\rho) + o(1) \text{ for all $\xi\in B(\xi_\infty, \frac 14\rho)$,}
$$
by \eqref{Equ:m-inf=m-infty=m}, which implies
$$
\strokedint_{B(\xi_\infty, \frac 14\rho)} \hat u_{r_k}(\xi) d\xi \leq \hat u_{r_k}(\xi_k) + \strokedint_{B(\xi_\infty, \frac 14\rho)} W^{\hat\mu_{r_k}}_{1,p}(\xi, 2\rho)d\xi
+ o(1)
$$
and, after taking $k\to\infty$,
$$
m^\varsigma  \strokedint_{B(\xi_\infty, \frac 12\rho)} |\xi|^{-\frac{n-p}{p-1}} d\xi \leq m + \varepsilon(\rho),
$$
where $\varepsilon(\rho)\to 0$ as $\rho\to 0$, provided that
\begin{equation}\label{Equ:average-wolff-good}
\strokedint_{B(\xi_\infty, \frac 14\rho)} W^{\hat\mu_{r_k}}_{1,p}(\xi, 2\rho)d\xi\to 0 \text{ as $k\to\infty$}.
\end{equation}
This implies $m^\varsigma \leq m$ when let $\rho\to 0$. Thus $m^\varsigma = m$. To complete this step we want to verify \eqref{Equ:average-wolff-good}. This
is because 
$$
\aligned
\int_{B(\xi_\infty, \frac 14\rho)}  & W^{\hat\mu_{r_k}}_{1,p}(\xi, 2\rho)d\xi 
= \int_{B(\xi_\infty, \frac 14\rho)} \int_0^{2\rho} \left(\frac {\hat\mu_{r_k} (B(\xi, t))}{t^{n-p}} \right)^\frac 1{p-1}\frac {dt}t\, d\xi\\
& \leq C\rho^\frac {n(p-2)}{p-1} \int_0^{2\rho} \frac 1 {t^\frac {n-1}{p-1}} (\int_{B(\xi_\infty, \frac 14\rho)}\int_{B(\xi, t)} d\hat\mu_{r_k}(\bar\xi)  \, d\xi)^\frac 1{p-1} \, dt \\
& \leq C \rho^\frac {n(p-2)}{p-1} \int_0^{2\rho}\frac 1 {t^\frac {n-1}{p-1}} (\int_{B(\xi_\infty, t+\frac 14\rho)}\int_{B(\bar\xi, t)} d\xi\, d\hat\mu_{r_k}(\bar\xi))^\frac 1{p-1}\, dt \\
& \leq C \rho^{n - \frac {n-p}{p-1}} (\hat\mu_{r_k} (B(\xi_\infty, \frac 94\rho) )^\frac 1{p-1}  = C\rho^{n - \frac {n-p}{p-1}} (\hat\mu (r_k B(\xi_\infty, \frac 94\rho))^\frac 1{p-1} \to 0 \text{ as $k\to\infty$}.
\endaligned
$$

\noindent{\bf Step 4.}\quad  In this step, based on Part 2, we use Theorem \ref{Thm:wolff potential upper bdd-singualarity}
and  \cite[Theorem 4.8]{KM94} to establish the limit as $x\to 0$ away from a set $E$ that is $p$-thin for singular behavior at the origin. 
First, in our notation, we may write
$$
|x|^\frac {n-p}{p-1}\hat u(x) = \hat u_{|x|} (\frac x{|x|}).
$$
From the lower bound \eqref{Equ:lower-bdd} established in Step 1, it suffices to get the estimates from above. By \cite[Theorem 4.8]{KM94}, we have, for $\epsilon>0$
and then appropriately small $\sigma \in(0, \frac 16)$
$$
\aligned
\hat u_{|x|}(\frac x{|x|}) - & (m-\varepsilon)  \leq c((\strokedint_{B(\frac x{|x|}, \sigma)}
|\hat u_{|x|}(\xi)  - (m-\varepsilon)|^qd\xi)^\frac 1q +  |x|^\frac {n-p}{p-1}W^{\hat\mu}_{1, p}(x, 2\sigma |x|)) \\
& \leq  c(\strokedint_{B(\frac x{|x|}, \sigma)} |m|\xi|^{-\frac{n-p}{p-1}}  - (m-\varepsilon)|^qd\xi)^\frac 1q + o(1) \text{ as $|x|\to 0$} \\
& \leq 2\varepsilon
\endaligned
$$
as $x\to 0$ and $x\notin E^{\hat\mu}$ that is $p$-thin for singular behavior at the origin, 
where $q\in (p-1, \frac {np}{n-p})$ is fixed and the argument for the forth term $W^{\hat\mu}_{1, p}(x, (1-\delta)|x|)$ in the proof of 
Theorem \ref{Thm:wolff potential upper bdd-singualarity} is applied to pick up the set $E^{\hat\mu}$ that is $p$-thin for singular
behavior at the origin. Let $E= E^\mu \bigcup E^{\hat\mu}$, which is still $p$-thin for singular behavior at the origin. Thus we are able to end this step and 
conclude
\begin{equation}\label{Equ:m-limit-verified}
\lim_{x\to 0\text{ and } x\notin E} |x|^\frac {n-p}{p-1}u(x) = m.
\end{equation}

\noindent{\bf Step 5.} \quad In this step, based on the rescaling approach in Step 2, Step 3, and \eqref{Equ:m-limit-verified}, we want to identify what is
$m$. For this purpose, let us consider $A_{s_0, \frac 1{s_0}}$ for a fixed small $s_0>0$ and 
$$
\phi\in C^\infty_c(B(0, \frac 1{s_0})) \text{ and } \phi \equiv 1 \text{ in a neighborhood $B(0, \epsilon_0)$ of the origin for some $\epsilon_0 < < s_0$}.
$$
First, for a sequence $r_k\to 0$, from the consequence of Lemma \ref{Lem:known-to-veron} from 
$$\hat u_{r_k}\to \hat u_\infty= m|\xi|^{-\frac {n-p}{p-1}}  \text{ strongly in $W^{1, p}_{\text{loc}}(\mathbb{R}^n\setminus \{0\})$ as $k\to\infty$}, 
$$
we calculate
$$
\aligned
\int_{A_{s_0, \frac 1{s_0}}} & (-\Delta_p \hat u_{r_k})\phi d\xi - \int_{A_{s_0, \frac 1{s_0}}} |\nabla\hat u_{r_k}|^{p-2}\nabla\hat u_{r_k}\cdot
\nabla\phi d\xi \\
& =  \int_{B(0, s_0)}  |\nabla\hat u_{r_k}|^{p-2}\nabla\hat u_{r_k}\cdot \nabla\phi d\xi - \int_{B(0, s_0)} (-\Delta_p \hat u_{r_k})\phi d\xi \\
& =  \int_{B(0, s_0)\setminus B(0, \epsilon_0)}  |\nabla\hat u_{r_k}|^{p-2}\nabla\hat u_{r_k}\cdot \nabla\phi d\xi - \int_{B(0, r_ks_0)}\phi(r_kx)d\hat \mu(x) 
\endaligned
$$
and
$$
\aligned
\int_{A_{s_0, \frac 1{s_0} }} & (-\Delta_p \hat u_\infty)\, \phi d\xi  - \int_{A_{s_0, \frac 1{s_0}}} |\nabla\hat u_\infty|^{p-2}\nabla\hat u_\infty
\cdot \nabla\phi d\xi \\
& =  \int_{B(0, s_0)\setminus B(0, \epsilon_0)}  |\nabla\hat u_\infty|^{p-2}\nabla\hat u_\infty\cdot \nabla\phi d\xi - \alpha_0(n, p)m^{p-1}.
\endaligned
$$
Notice that, taking $k\to\infty$, we have
$$
\int_{B(0, r_ks_0)}\phi(r_kx)d\hat \mu(x)\to\hat\mu(\{0\}).
$$
Therefore
$$
\alpha_0(n, p) m^{p-1} = \hat\mu(\{0\}).
$$
Therefore, to complete the proof of Theorem \ref{Thm:main-asymptotic}, it suffices to use Lemma \ref{Lem:point-charge} below.

\end{proof}

\begin{lemma}\label{Lem:point-charge}
$$
\mu(\{0\}) = \hat \mu(\{0\}).
$$
\end{lemma}
\begin{proof} Let
$$
u_r(\xi) = \frac {u(r\xi)}{r^{-\frac {n-p}{p-1}}} \ \text{ and } \hat u_r(\xi) = \frac {\hat u(r\xi)}{r^{-\frac {n-p}{p-1}}}.
$$
Then, recall from Step 2,
$$
-\Delta_p^\xi \hat u_r(\xi) = \hat \mu_r
$$
and
$$
\hat u_{r} (\xi) \to \hat u_\infty (\xi), \  \hat\mu_{r} \to \hat\mu(\{0\})\delta_0 \ \text{ and } -\Delta_p^\xi \hat u_\infty = \hat\mu(\{0\})\delta_0
$$
as $r\to 0$. Similarly, we also have
$$
-\Delta_p^\xi u_r  = \mu_r \text{ and } \mu_r \to \mu(\{0\}) \delta_0  \text{ as $r\to 0$}
$$ 
where $d\mu_r (\cdot) = r^n d\mu(r\cdot)$. Therefore it suffices to show 
\begin{equation}\label{Equ:same-limit}
\lim_{r\to 0} \mu_r = \lim_{r\to 0} \hat\mu_r.
\end{equation}
For a testing function $\phi$, we have
$$
\aligned
\mu_r (\phi) -\hat\mu_r(\phi) & = \int ( |\nabla u_r|^{p-2}\nabla u_r \cdot \nabla\phi - |\nabla \hat u_r|^{p-2}\nabla \hat u_r \cdot \nabla\phi) \\
& \leq C (\int_{\frac 1r E^\mu}|\nabla \phi|^\frac {p-1}{p-2})^\frac {p-2}{p-1}\\ 
\endaligned
$$
because that $u_r$ and $\hat u_r$ are identical outside $\frac 1r E^{\mu}$ and that
\begin{equation}\label{Equ:l^{p-1}-bound}
\|\nabla \hat u_r\|_{L^{p-1} (\text{supp}\phi)} + \|\hat u_r\|_{L^{p-1}(\text{supp}\phi)} \leq C,
\end{equation}
for some $C$ that is independent of $r$, due to Lemma \ref{Lem:l^p-estimate}.
Recall that $E^\mu$ is $p$-thin for the singular behavior at the origin according to \eqref{Equ:C-hat}.
\\

Let 
\begin{equation}
E_j = 2^j E^\mu \text{ and }E_\infty = \bigcap_{i=1}^\infty \bigcup_{j\geq i} E_j.
\end{equation}
Then we have only need to show that the Lebesgue measure of $E_\infty$ vanishes.
In fact, we want to show that $E_\infty$ is of zero $p$-capacity due to the fact that $E^\mu$ is $p$-thin for singular behavior at the origin. This is because
\begin{equation}
\bigcup_{j\geq 1}E_j\cap\bar A_{2^{-k}, 2^k} \supset \bigcup_{j\geq 2} E_j\cap \bar A_{2^{-k}, 2^k} \cdots\supset\bigcup_{j\geq i} 
E_j\cap \bar A_{2^{-k}, 2^k} \cdots \to E_\infty\cap\bar A_{2^{-k}, 2^k}
\end{equation}
where each $E_j$ is closed due to the lower semi-continuity of $p$-superharmonicity and
\begin{equation}
\aligned
cap_p( \bigcup_{j\geq i} & E_j\cap\bar A_{2^{-k},  2^k}, \mathbb{R}^n)  \leq \sum_{l=-k+1}^{k} \sum_{j\geq i} cap_p(E_j \cap \omega_l, \Omega_l) \\
&  \leq  \sum_{l=-k+1}^k 2^{(n-p)l}\sum_{j\geq i}^{\infty} \frac{cap_p(E\cap \omega_{l+j}, \Omega_{l+j})}{cap_p(\omega_{l+j}, \Omega_{l+j})}\\
& \leq  \sum_{l=-k+1}^k 2^{(n-p)l}\sum_{j\geq i+l}^{\infty} \frac{cap_p((E\cap \omega_{j}, \Omega_{j})}{cap_p(\omega_{j}, \Omega_{j})} \to 0
\endaligned
\end{equation}
as $i\to\infty$ for any fixed $k$ according to Definition \ref{Def:quasi-p-thin}. Therefore, by \cite[Theorem 2.2 (iv)]{HKM},
$$
cap_p(E_\infty\cap\bar A_{2^{-k}, 2^k},  \mathbb{R}^n) =\lim_{i\to\infty} cap_p (\bigcup_{j\geq i} E_j\cap\bar A_{2^{-k},  2^k}, A_{2^{-k-1}, 2^{k+1}},  \mathbb{R}^n) =0
$$
for each fixed $k$, which implies $cap_p(E_\infty,  \mathbb{R}^n) = 0$. Thus, according to \cite[Lemma 2.10]{HKM}, for instance, 
the Lebesgue measure of $E_\infty$ vanishes. This completes the proof.
\end{proof}


\subsection{Applications to fully nonlinear elliptic equations}\label{Subsec:fully-nonlinear-applications}
 
In this subsection we collect some corollaries of Theorem \ref{Thm:main-asymptotic} on the solutions to fully nonlinear elliptic equations.
It is interesting to compare the intermediate positivity cones $\mathcal{A}^{(p)}$ with those in the study of fully nonlinear 
equations. Recall
$$
\Gamma^k = \{\lambda = \{(\lambda_1, \lambda_2, \cdots, \lambda_n)\in \mathbb{R}^n: \sigma_1(\lambda)\geq 0, \sigma_2(\lambda)\geq 0, \cdots, \sigma_k(\lambda)\geq 0\}
$$
for $k=1, 2, \cdots, n$, where $\sigma_l$ is the elementary symmetric functions for $l=1, 2, \cdots, n$. It is easily seen that $\mathcal{A}^{(2)} =\Gamma^1$ 
and $\mathcal{A}^{(p)}$ approaches $\Gamma^n$ as $p\to\infty$. Hence, for any positive cone $\Gamma$ between $\Gamma^1$ and 
$\Gamma^n$, we may consider
\begin{equation}\label{Equ:p-index}
p_\Gamma = \max \{p: \Gamma \subset \mathcal{A}^{(p)}\}.
\end{equation}
$p_\Gamma$ is useful when one uses $p$-superharmonic functions to study solutions to a class of fully nonlinear elliptic equations. We first realize

\begin{lemma}\label{Lem:fully-p-gamma} Suppose that $u$ is nonnegative and that $u\in C^2 (\Omega\setminus S)$ for a compact subset $S$ of a
bounded domain $\Omega$ in $\mathbb{R}^n$. And suppose $\lim_{x\to S}u(x) = +\infty.$ 
Assume -$\lambda (D^2 u(x)) \in \Gamma$ for $p_\Gamma \in (1, n]$. Then $u$ is a $p_\Gamma$-superharmonic function in $\Omega$.
\end{lemma}

\begin{proof} By the assumptions $u$ is lower semi-continuous and clearly $u\not\equiv+\infty$ in $\Omega$, if we set $u=+\infty$ on $S$. 
Then, according to Definition \ref{Def:p-superhar}, it
suffices to check the comparison principle. For any domain $D\subset\subset \Omega$,  let $v$ be a $p_\Gamma$-harmonic function in $D$ and $v\leq u$ on
$\partial D$. We want to verify that $v\leq u$ in $D$. Assume otherwise, then there is a domain $D_1\subset D$, where $v> u$ in $D_1$ and $v = u$ on 
$\partial D_1$.  It is easily seen that $\bar D_1\subset \Omega \setminus S$, which violates the comparison principle for $p_\Gamma$-superharmonic 
functions on $D_1$. 
\end{proof}

We remark that, from the proof of \cite[Lemma 3.2 and 3.3]{MQ021} (see also \cite[Proposition 1.1]{BV89} when $S$ is an isolated point), 
it is easily seen that $-\Delta_{p_\Gamma} u$ is a Radon measure under even somewhat weaker assumptions. Consequently, we have
 
\begin{corollary}\label{Cor:app-fully-nonlinear} 
Suppose that $u$ is nonnegative and that $u\in C^2 (\Omega\setminus S)$ for a compact subset $S$ of a bounded domain $\Omega$ in $\mathbb{R}^n$. 
And suppose $\lim_{x\to S}u(x) = +\infty.$ 
Assume -$\lambda (D^2 u(x)) \in \Gamma$ for $p_\Gamma \in (1, n]$. Then $S$ is of Hausdorff dimension not greater than $n-p_\Gamma$ and, 
for $x_0\in S$, there are a subset $E$ that is $p_\Gamma$-thin 
for the singular behavior at $x_0$ and a nonnegative number $m$ such that
$$
\lim_{x\to x_0 \text{ and } x\notin E}  \frac {u(x)}{G_{p_\Gamma}(x, x_0)}= m.
$$
Moreover $u(x) \geq m G_{p_\Gamma}(x, x_0) - c_0$ in some neighborhood of $x_0$.
\end{corollary}

It is well-known that the polar set of a $p$-superharmonic function is of zero $p$-capacity (see, for example, \cite[Theorem 10.1]{HKM}). Therefore, in the light of 
\cite[Lemma 2.1]{MQ021} (see also \cite{AM72, SY88}), the polar set of a $p$-superharmonic function is of Hausdorff dimension not greater than $n-p$.
It turns out we have a rather effective way to calculate $p_\Gamma$ for a cone associated with a homogeneous, symmetric, convex function of $n$-variables.

\begin{lemma}\label{Lem:calculate-p-index} Suppose that $\Gamma$ is a cone given by a homogeneous, symmetric, convex function $F(\lambda)$ on 
$\mathbb{R}^n$. Let $(-\frac {n-1}{p-1}, 1, 1, \cdots, 1)\in \partial\Gamma = \{\lambda\in \mathbb{R}^n: F(\lambda)=0\}$. Then 
$\Gamma\subset\mathcal{A}^{(p)}$ and $p_\Gamma = p$.
\end{lemma}

\begin{proof} Due to the symmetric property with respect to the variables $\lambda = (\lambda_1, \lambda_2, \cdots, \lambda_n)$, we only need to 
verify that the cone $\Gamma$ is supported by the hyperplane 
$$
\mathcal{P} = \{\lambda\in \mathbb{R}^n: (p-1)\lambda_1 + \lambda_2 +\cdots +\lambda_n=0\},
$$ 
which clearly supports the cone $\mathcal{A}^{(p)}$ and is with the normal vector $(p-1, 1, 1, \cdots, 1)$. By the assumption, we know that
$$
F( - \frac {n-1}{p-1}t, t, \cdots, t) = 0
$$
for all $t\geq 0$. To verify that $\mathcal{P}$ supports $\Gamma$ along the ray $\mathcal{R} = t(-\frac {n-1}{p-1}, 1, 1, \cdots, 1)$ is to check if the 
normal direction of $\Gamma$ along the ray $\mathcal{R}$ is the same as that of $\mathcal{P}$. After differentiation, we derive
$$
(-\frac {n-1}{p-1}\partial_{\lambda_1}F + \sum_{i=2}^n\partial_{\lambda_i}F )|_{(-\frac {n-1}{p-1}t, t, \cdots, t)} 
= (-\frac {n-1}{p-1}\partial_{\lambda_1}F + (n-1)\partial_{\lambda_k}F )|_{(-\frac {n-1}{p-1}t, t, \cdots, t)} = 0
$$
for $k=2, 3, \cdots, n$, which yields that $\nabla F$ to $\Gamma$ along the ray $\mathcal{R}$ is indeed in the same direction of the normal vector $(p-1, 1, 1, \cdots, 1)$
to the hyperplane $\mathcal{P}$. With the ray $\mathcal{R}$ in the boundary of both cones it is obvious that $p_\Gamma = p$. The proof is therefore
completed.
\end{proof}

Consequently, we can calculate $p_{\Gamma^k}$ easily. 

\begin{corollary}\label{Cor:gamma-k-p-index} For the positive cone $\Gamma^k$, we have
\begin{equation}\label{Equ:gamma-k-p-index}
p_{\Gamma^k} = \frac {n(k-1)}{n-k} + 2 \in [2, n]
\end{equation}
for $1\leq k \leq \frac n2$.
\end{corollary}
\begin{proof} This simply is because
$$
\sigma_k(-\frac {n-k}k, 1, 1, \cdots, 1) = -\frac {n-k}k \left(\aligned k-1\\n-1\endaligned\right) + \left(\aligned & k\\n & -1\endaligned\right) = 0.
$$
\end{proof}  
 
Remarkably, we are able to derive an asymptotic estimates that extends \cite[Theorem 3.6]{Lab02} significantly.

\begin{corollary}\label{Cor:sigma-k-asymptotic} Suppose that $u$ is nonnegative and that $u\in C^2 (\Omega\setminus S)$ for a
compact subset $S$ inside a bounded domain $\Omega$ in $\mathbb{R}^n$. And suppose 
$\lim_{x\to S}u(x) = +\infty.$
Assume $-\lambda (D^2 u(x)) \in \Gamma^k$ for $ 1\leq k \leq \frac n2$. Then $S$ is of Hausdorff dimension not greater than $n-p_\Gamma$ and, 
for $x_0\in S$, there are a subset $E$ that is $p_{\Gamma^k}$-thin 
for the singular behavior at $x_0$ and a nonnegative number $m$ such that
$$
\lim_{x\to x_0 \text{ and } x\notin E}  \frac {u(x)}{\mathcal{G}^k(x, x_0)} = m.
$$
Moreover $u(x) \geq m \mathcal{G}^k(x, x_0) - c_0$ in some neighborhood of $x_0$, where
$$
\mathcal{G}^k(x, x_0) = \left\{\aligned |x-x_0|^{2 - \frac nk} & \text{ when } 1\leq k < \frac n2\\
- \log |x-x_0| & \text{ when } k = \frac n2\endaligned\right..
$$
\end{corollary}

In general it takes a lot more to rule out the thin set $E$ in the above two corollaries, even for isolated singularities (cf. \cite{KV, Vn17, Lab02}).


\vskip 0.3cm
\noindent Huajie Liu: Department of Mathematics, Nankai University, Tianjin, China; \\e-mail: 
1120220031@mail.nankai.edu.cn 
\vspace{0.2cm}

\noindent Shiguang Ma: Department of Mathematics, Nankai University, Tianjin, China; \\e-mail: 
msgdyx8741@nankai.edu.cn 
\vspace{0.2cm}

\noindent Jie Qing: Department of Mathematics, University of California, Santa Cruz, CA 95064; \\
e-mail: qing@ucsc.edu \\
Research of this author is partially supported by Simons Foundation.
\vspace{0.2cm}

\noindent Shuhui Zhong: School of Mathematics, Tianjin University, Tianjin, China; \\
e-mail: zhshuhui@tju.edu.cn



\begin{thebibliography}{999999}
\bibliographystyle{plain}

\footnotesize


\bibitem[AH96]{AH96} David R. Adams and Lars I. Hedberg, {\it Function spaces and potential theory}, Springer-Verlag, Berlin Heidelberg, 1996.

\bibitem[AM72]{AM72} David R. Adams and Norman G. Meyers, {\it Thinness and Wiener criteria for non-linear potentials}, Indiana Univ. Math. J. 
{\bf 22} (1972), 169 - 197.


\bibitem[AH73]{AH73} Maynard Arsove and Alfred Huber, {\it Local behavior of subharmonic functions}, Indiana Univ. Math. J. {\bf 22} (1973), 
1191 - 1199.


\bibitem[BV89]{BV89} Marie-Francois Bidaut-Veron, {\it Local and global behavior of solutions of quasi-linear equations of emden-fowler type},
Arch. Rational Mech. Anal. {\bf 107} (1989), no. 4, 293 - 324..



\bibitem[Br40]{Br40} Marcel Brelot, {\it Points irreguliers et transformations continues en th\"{e}orie du potentiel}, J. de Math. 
{\bf 19} (1940), 319 - 337.

\bibitem[Br44]{Br44} Marcel Brelot, {\it Sur les ensembles effil\'{e}s}, Bull. Sci. Math. {\bf 68} (1944), 12 - 36.

\bibitem[B1873]{Rey} Paul du Bois-Reymond, {\it Eine neue Theorie der Convergenz und Divergenz von Reihen mit positiven 
Gliedern}, Journal für die reine und angewandte Mathematik {\bf 76} (1873), 61 - 91.

\bibitem[B1908]{Brom} Thomas Bromwich, {\it Introduction to the theory of infinite series}, The Macmillan, New York, 1908.












\bibitem[HK88]{HeKi}  Juha Heinonen and Tero Kilpel\"{a}inen,  {\it On the Wiener criterion and quasilinear obstacle problems}. 
Trans. Amer. Math. Soc. {\bf 310} (1988), 239 - 255.

\bibitem[HKM93]{HKM} Juha Heinonen, Tero Kilpel\"{a}inen, and Olli Martio, {\it Nonlinear potential theory of degenerate elliptic equations}, 
Oxford Univ. Press, Oxford, 1993.

\bibitem[Hu57]{Hu57} Alfred Huber, {\it On subharmonic functions and differential geometry in the large},
Comment. Math. Helv. {\bf 32} (1957), 13 - 72.



\bibitem[KV86]{KV} Satyanad Kichenassamy and Laurent Veron, {\it Singular solutions of the p-Laplace equation}, 
Math. Ann. {\bf 275} (1986), no. 4, 599 - 615.

\bibitem[KV87c]{KV-e} Satyanad Kichenassamy and Laurent Veron, {\it Erratum: "Singular solutions of the p-Laplace equation'' }
Math. Ann. {\bf 277} (1987), no. 2, 352.

\bibitem[KM92]{KM92} Tero Kilpel\"{a}inen and Jan Mal\'{y}, {\it Degenerate elliptic equations with measure data and nonlinear potentials}, 
Ann. Scuola Norm. Sup. Pisa Cl. Sci. (4) {\bf 19} (1992), no. 4, 591 - 613. 

\bibitem[KM94]{KM94} Tero Kilpel\"{a}inen and Jan Mal\'{y}, {\it The wiener test and potential estimates for quasilinear elliptic equations}, 
Acta. Math. {\bf 172} (1994), 137 - 161.


\bibitem[Lab02]{Lab02} Denis Labutin, {\it Potential estimates for a class of fully nonlinear elliptic equations}, Duke Math J. {\bf 111} (2002),
no. 1, 1 - 49.




\bibitem[Lind06]{Lind06} Peter Lindqvist, {\it Notes on the $p$-Laplace equation},  University of Jyvaskyla Lecture Notes 2006.


\bibitem[MQ21a]{MQ21} Shiguang Ma and Jie Qing, {\it On n-superharmonic functions and some geometric applications}, 
Calc. Var. Partial Differential Equations {\bf 60} (2021), no. 6, Paper No. 234, 42 pp.

\bibitem[MQ21b]{MQ021} Shiguang Ma and Jie Qing,  {\it On Huber-type theorems in general dimensions},
 Adv. in Math. {\bf 395} (2022), Paper No. 108145, 37 pp. 
 
 \bibitem[MQ22]{MQ22} Shiguang Ma and Jie Qing, {\it Linear potentials and applications in conformal geometry}, preprint 2022,
arXiv:2209.02823.

\bibitem[LMQZ]{MQ23} Huajie Liu, Shiguang Ma, Jie Qing, Shuhui Zhong, {\it $p$-Laplace equations in conformal geometry}, 
preprint 2023 arXiv:2308.02468.

\bibitem[Mi96]{Mi96} Yoshihiro Mizuta, {\it Potential theory in Euclidean spaces}, Gakuto International Series, Mathematical Sciences
and Applications, {\bf 6}, (1996), Gakk\-{o}tosho Co., Ltd., Tokyo, Japan.


\bibitem[PV08]{PV08} Nguyen Cong Phuc and Igor Verbitsky, {\it Quasilinear and hessian equations of lane-emden type}, 
Ann. Math. {\bf 168} (2008), 859 - 914.



\bibitem[SY88]{SY88}  Richard Schoen and S. T. Yau, {\it Conformally flat manifolds, Kleinian groups and scalar curvature}, 
Invent. Math. {\bf 92} (1988), 47 - 71. 

\bibitem[Tol83]{Tol83} P. Tolksdorf, {\it On the Dirichlet problem for quasilinear equations in domain with conical boundary points},
Commun. Partial. Differ. Equations, {\bf 8} (1983), 773 - 817.

\bibitem[Vn17]{Vn17} Laurent V\'{e}ron, {\it Local and Global Aspects of Quasilinear Degenerate Elliptic Equations}, New Jersey, 
World Scientific, 2017. 




\end{thebibliography}
\end{document}